\documentclass[12pt]{amsart}
\usepackage{amsthm,amsmath,amssymb,amscd,graphics,enumerate, stmaryrd,xspace,verbatim, epic, eepic,url,fullpage}

\usepackage[usenames,dvipsnames]{color}

\usepackage[colorlinks=true]{hyperref}

\usepackage[active]{srcltx}
\usepackage[all]{xypic}
\SelectTips{cm}{}

{
   
      \newtheorem*{theorem*}{Theorem}
   \newtheorem{proposition}[subsubsection]{Proposition}

   \newtheorem{lemma}[subsubsection]{Lemma}

   \newtheorem*{conjecture*}{Conjecture}

}
{\theoremstyle{definition}
          \newtheorem*{exercise*}{Exercise}

   \newtheorem*{example*}{Example}

   \newtheorem*{definition*}{Definition}

}
\newcommand{\QQ}{{\mathbb{Q}}}

\renewcommand{\cD}{{\mathcal D}}

\newcommand{\cF}{{\mathcal F}}

\newcommand{\cI}{{\mathcal I}}
\newcommand{\cJ}{{\mathcal J}}

\newcommand{\cO}{{\mathcal O}}

\def\<{\langle}
\def\>{\rangle}

\newcommand{\Spec}{\operatorname{Spec}}

\newcommand{\Div}{{\operatorname{Div}}}

\def\:{{\colon}}
\def\.{{,\dots,}}

\newcommand{\double}{\genfrac..{0pt}1
{\raise -1pt\hbox{$\scriptstyle\longrightarrow$}}{\raise 3pt\hbox
{$\scriptstyle\longrightarrow$}}}

\def\sat{{\rm sat}}

\def\int{{\rm int}}

\def\tototi{\mathbin{\mathop{\otimes}\limits^{\raise-1pt\hbox
{$\scriptscriptstyle {\rm L}$}}}}

\def\indlim{\mathop{\vrule width0pt height7pt depth
4pt\smash{\lim\limits_{\raise 1pt\hbox to 14.5pt
{\rightarrowfill}}}}}
\def\projlim{\mathop{\vrule width0pt height7pt depth
4pt\smash{\lim\limits_{\raise 1pt\hbox to 14.5pt
{\leftarrowfill}}}}}

\newcommand\displaceamount{3pt}

\newcommand{\doubledown}{\ar@<\displaceamount>[d]\ar@<-\displaceamount>[d]}

\newcommand{\doubleup}{\ar@<\displaceamount>[u]\ar@<-\displaceamount>[u]}

\newcommand{\doubleright}{\ar@<\displaceamount>[r]\ar@<-\displaceamount>[r]}

\def\gp{\text{gp}}

\begin{document}
\title{Functorial monomialization and uniqueness of centers for relative principalization}

\author[D. Abramovich]{Dan Abramovich}
\address{Department of Mathematics, Box 1917, Brown University,
Providence, RI, 02912, U.S.A}
\email{dan\_abramovich@brown.edu}

\author[M. Temkin]{Michael Temkin}
\address{Einstein Institute of Mathematics\\
               The Hebrew University of Jerusalem\\
                Edmond J. Safra Campus, Giv'at Ram, Jerusalem, 91904, Israel}
\email{temkin@math.huji.ac.il}

\author[J. W{\l}odarczyk] {Jaros{\l}aw W{\l}odarczyk}
\address{Department of Mathematics, Purdue University\\
150 N. University Street,\\ West Lafayette, IN 47907-2067}
\email{wlodar@math.purdue.edu}


\thanks{This research is supported by BSF grants 2018193 and 2022230, ERC Consolidator Grant 770922 - BirNonArchGeom, NSF grants 
DMS-2100548, DMS-2401358,  and Simon foundation grants MPS-SFM-00006274, MPS-TSM-00008103. Dan Abramovich would like to thank IHES and the Hebrew University of Jerusalem for welcoming environments during periods of preparation of this article.}

\date{\today}

\begin{abstract}
Theorem 1.2.6 of \cite{ATW-relative} provides a \emph{relatively} functorial logarithmic principalization of ideals on \emph{relative logarithmic orbifolds} $X \to B$  in characteristic 0, relying on a delicate monomialization theorem for Kummer ideals. The paper \cite{ABTW-foliated} provides a parallel  avenue through weighted blowings up. In this paper we show that, \emph{if $X \to B$ is proper,} monomialization of both Kummer and weighted logarithmic centers can be carried out in a manner which is functorial for base change by regular morphisms. This implies in particular logarithmic relative principalization of ideals and logarithmically smooth reduction of proper families of varieties in characteristic 0 in a manner equivariant for group actions and for localization on the base.
\end{abstract}
\maketitle


\section{Introduction}
\subsection{Setup}

Consider a relative logarithmic orbifold  $X \to B$ as in \cite{ATW-relative} in characteristic 0. This means $X$ and $B$ are logarithmically regular Deligne--Mumford stacks and $X\to B$ is a logarithmically regular morphism with enough derivations. The key case is where $X,B$ are toroidal varieties, and the morphism $X \to B$ is toroidal.

We are given an ideal sheaf $\cI \subset \cO_X$ which we wish to principalize.

We assume $X \to B$ is an integral logarithmic morphism.
We also assume $B$ is regular.

Write $\cF= \cD_{X/B}$, the sheaf of relative \emph{logarithmic} derivations. The results of  \cite{ATW-relative} 
 provide local parameters $t_i$ and ideal $\cI_\infty$ defined on the vanishing locus $V(t_1,\ldots,t_k)$, bundled as an object
locally of the form $J:=(t_1,\ldots,t_k, \cI_\infty^{1/d})$, of maximal invariant $(a_1,\ldots,a_k,\infty)$; here $\cI_\infty$ has the property that $\cF'^{\leq 1}(\cI_\infty ) =\cI_\infty$, where $\cF'$ is the sheaf of relative logarithmic derivations on $V(t_1,\ldots,t_k)$.

Section 3 of  \cite{ATW-relative}, provides a logarithmic blowup $B' \to B$ over which  $\cI_\infty$ is monomial, making $J$ into a \emph{submonomial} center, a well-defined \emph{Kummer ideal} that can be blown up. That monomialization construction is decidedly not functorial, as it involves choices on several junctures. A related issue is that in \cite{ATW-relative} the object $J$ is not provided an intrinsic meaning, in particular, not shown to be uniquely determined by $\cI$ as such. The parameters $t_i$ are rarely unique.

 The work \cite{ABTW-foliated}, especially Section 5, replaces $J$ by a weighted center $J^w=(t_1^{a_1},\ldots, t_k^{a_k}, \cI_\infty^{1/d})$, along the way to replacing $\cI_\infty^{1/d}$ by its weighted center $(s_1^{b_1},\ldots, s_l^{b_l})$.  The same monomialization theorem applies.
The independence of $J^w$ of choices is again not addressed in \cite{ABTW-foliated}, as only the final logarithmic center $(t_1^{a_1},\ldots, t_k^{a_k},\ s_1^{b_1},\ldots,s_l^{b_l})$ is discussed: that is a well-defined $\QQ$-ideal associated to $\cI$, independent of the choices made in its construction.

By \cite[Proposition 4.5.6]{ABTW-foliated} there is a formal lifting of $\cI_{\infty}$ from $V(t_1,\ldots,t_k)$ to $X$ so that $\cF^{\leq 1}(\cI_{\infty}) = \cI_{\infty}$, giving a well-defined $\QQ$-ideal $J$, respectively $J^w$.  In fact $J$ requires only a lifting of $\cI_{\infty}$ modulo $(t_1,\ldots,t_k)^d$, whereas $J^w$ requires only a lifting modulo $(t_1,\ldots,t_k)^{d\cdot a_k}$. We choose such a lifting. Note that so far these $\QQ$-ideals $J$ and $J^w$  depend both on the choices of $t_i$ and this lifting.


\subsection{Results}
The purpose here is to make this monomialization step as  functorial as possible:  if $\tilde B \to B$ is a \emph{regular morphism,} with pullbacks $\tilde X$ and $\tilde J$, then  the morphism $\tilde B' \to \tilde B$ has $\tilde B' = B' \times_B \tilde B$.

We define the \emph{principality locus} of $\cI_\infty$ for $X \to B$ to be the complement of the image in $B$ of the locus  where $\cI_{\infty}$ is not principal. 
We have:
\begin{proposition} \label{Prop:monomialization}
Let $X \to B$ be an \emph{integral and proper} relative logarithmic orbifold, with $B$ regular. Consider a $\QQ$-ideal $J$, respectively $J^w$, locally of the form $$J=(t_1,\ldots,t_k, \cI_\infty^{1/d}), \qquad \text{ respectively }\qquad  J^w=(t_1^{a_1},\ldots, t_k^{a_k}, \cI_\infty^{1/d}),$$ with $\cF^{\leq 1}(\cI_\infty) =\cI_\infty$. There is a functorial birational morphism $B' \to B$, with saturated pullback $X' = (X \times_BB')^\sat$, so that $J':= J\cO_{X'}$, respectively  $J'^w:= J^w\cO_{X'}$ is submonomial. The blowup $B'\to B$ is an isomorphism on the principality locus of  $\cI_{\infty}$ for $X \to B$.
\end{proposition}
The key for this is functorial flattening. To address the choices  we have:
\begin{proposition}\label{Prop:uniqueness} The object $J=(t_1,\ldots,t_k, (\cI_\infty)^{1/d})$, respectively $J^w=(t_1^{a_1},\ldots, t_k^{a_k}, \cI_\infty^{1/d})$, constructed for an ideal $\cI \subset \cO_X$  in \cite{ATW-relative}, respectively \cite{ABTW-foliated}, is uniquely defined as a $\QQ$-ideal, and is independent of the choices of parameters and liftings. \end{proposition}

As we see below, what makes this possible is that after any monomialization of $J$, its  Kummer blowup is independent of choices --- a property shown in \cite{ATW-relative}. Similarly, for a monomialized $J^w$ we have that its log weighted blowup is well defined, as shown in \cite{Quek}.

\subsection{Remarks}

We note that our proof of Proposition \ref{Prop:uniqueness} relies on the existence of monomialization. This is already proved in \cite[Section 3]{ATW-relative}, but also follows from  Proposition \ref{Prop:monomialization} if $X \to B$ is proper; for general $X\to B$ one can instead use any flattening procedure in the proof of Proposition \ref{Prop:monomialization} as functoriality is not needed  for Proposition \ref{Prop:uniqueness}. It is likely that a more direct argument can be obtained by thoroughly revisiting the structures introduced in \cite{ATW-relative}, but this is not our goal here.


Properness is a key assumption in the proof of Proposition \ref{Prop:monomialization}, as we use functorial flattening results that hold for proper morphisms.

Since we use flattening, if we start with a non-integral $X \to B$ it is advantageous  to first make it flat, namely integral. Assuming integrality from the outset allows that modifications in the process only touch loci of $B$ lying under the support in $X$ of the $\QQ$-ideal $J$  --- the locus where $\cI_\infty$ must be modified anyway. It also allows us to use the criteria in \cite[Sections 3.3 and 3.4]{ATW-relative}. The same holds for the assumption that $B$ is regular.

We note that a logarithmically smooth morphism $X \to B$, where $X,B$ are log regular and saturated \emph{and $B$ is regular,} is integral if and only if it is equidimensional. This is a consequence of miracle flatness, as in \cite[Remark 4.6]{AK}. In \cite[Section 4]{AK} it is shown how to make $X \to B$ equidimensional using  canonical combinatorial moves: one subdivides the complexes $\Sigma_X$ and $\Sigma_B$ so that the image of every cone in $\Sigma_X$ is a cone of $\Sigma_B$; after this one can use functorial resolution to make $B$ regular. Thus the integrality assumption is easy to satisfy as long as we make $B$ regular.


\section{Uniqueness of non-weighted centers}

We start the proof of Proposition \ref{Prop:uniqueness} by addressing the submonomial case:

\begin{lemma} Assume $J$ or $J^w$ is submonomial. Then it is uniquely defined as a $\QQ$-ideal, independent of the choices of parameters. Its formation commutes with saturated modifications $B' \to B$ and regular morphisms $\tilde{B} \to B$.\end{lemma}
\begin{proof} It is shown in \cite[Theorem 1.2.6]{ATW-relative} that there is a functorial  Kummer blowup $\pi: X' \to X$ with well-defined exceptional ideal sheaf $\cJ'$, both compatible with base change $B' \to B$, and depending only on the given ideal $\cI$ giving rise to $J$. The algebra $\oplus \pi_*\cJ'^n$ on $X$
is compatible with base change. It is finitely generated since it coincides with that of $J$, hence $J$ is well-defined as a $\QQ$-ideal. For $J^w$ use  \cite[Lemma 6.1]{Quek} instead.
\end{proof}

\begin{proof}[Proof of    Proposition \ref{Prop:uniqueness}] By \cite[Section 3]{ATW-relative} there is a modification $B' \to B$, with saturated pullback $X' \to B'$, so that the object $J' = J\cO_X'$ is submonomial; the argument uses $\cI_\infty$ only so it applies to $J^w$. Alternatively, when $X \to B$ is proper this follows from Proposiotion \ref{Prop:monomialization}.

By the lemma, $J'$ is a well-defined $\QQ$-ideal, and it is compatible with further modifications of $B'$, so we might as well choose such $B'$. Let $R'$ be the Rees algebra of $J'$. The push-forward $R$ of $R' $ in $X$ is therefore a well-defined algebra of ideals. It is finitely generated since it coincides with the algebra of the original $J$. Hence it gives a $\QQ$-ideal.
\end{proof}

\section{Monomialization}
We treat the unweighted and weighted cases in the same way, so in this section we write $J$ for either $J$ or $J^w$.
\subsection{A Kummer cover} The morphism $X \to B$ is integral but not saturated, and monomiality is easier to handle with saturated morphisms. On the other hand, \cite[Lemma 3.1.6]{ATW-relative} says that if we have a Kummer cover $B_1 \to B$ where the map $X_1 := (X \times _B B_1)^\sat \to B$ is saturated and $J_1$ is submonomial then $J$ is submonomial. So it is to our advantage to work on such $B_1$.

By \cite[Theorem 1.0.1]{Molcho} there is a functorial Kummer-\'etale and birational morphism $B_1 \to B$ of stacks so that the \emph{saturated} base change $X_1 = (X \times_{B_1} B)^\sat \to B_1$ is a saturated morphism. This stack theoretic covering avoids the choice of a Kawamata covering, which is necessarily not functorial,  in the  treatment of \cite[Proposition 5.1]{AK}.



We obtain a diagram of fs log regular stacks

\begin{equation} \label{saturation}\xymatrix{X_1 \ar[d]\ar[dr] \\ B_1\ar[dr] & X\ar[d] \\ & B,}\end{equation}
where $X_1 \to B_1$ is a saturated morphism, $B_1 \to B$  and $X_1 \to X$ are   Kummer morphisms, and the morphism $X \to B$ is integral but not necessarily saturated. By construction $B_1 \to B$ is a coarsening morphism; the same is true for $X_1 \to X$ but will not be necessary in our arguments.

If $B'' \to B$ is a birational  morphism of regular  logarithmic orbifolds, we denote by
\begin{equation} \label{pullback}\xymatrix{X''_1 \ar[d]\ar[dr] \\ B''_1\ar[dr] & X''\ar[d] \\ & B''}\end{equation}
the saturated pullback of diagram \eqref{saturation}, so again $X''_1 \to B''_1$ is a saturated morphism, $B''_1 \to B''$  and $X''_1 \to X''$ are   Kummer morphisms, and the morphism $X'' \to B''$ is integral but not necessarily saturated, and $B''_1 \to B''$ is a coarsening morphism.

\begin{lemma}\label{Lem:flattening}
There exists a functorial representable modification $B'' \to B$  of regular logarithmic orbifolds, such that in diagram \eqref{pullback} $J\cO_{X_1''}$ is \emph{flat} and $\cO_{X_1''} / J\cO_{X_1''}$ has \emph{monomial} torsion in pure codimension 1. The morphism $B''\to B$ is an isomorphism over the principality locus of $\cI_\infty$  for $X \to B$.
\end{lemma}

The proof proceeds by first flattening $J\cO_{X_1}$ over $B_1$ and then resolving singularities to retain the logarithmic structure and make the torsion monomial.

%
%


\subsection{Flattening}
The object $J\cO_{X_1}$ is represented by its Rees algebra  $R_J:=\oplus_{i\geq 0} J_i$, a quasi-coherent $\cO_{X_1}$-algebra generated in finitely many degrees by the coherent components $J_i$.

By Hironaka or Raynaud--Gruson, Rydh, McQuillan, see in particular  \cite[Theorem 1.2]{McQuillan-flat} or \cite[Theorem 1.1]{Rydh-flattening}, there is a functorial flattening of $\Spec_{X_1}R_J$ on $X_1 \to B_1$:  a blowup $B_1'\to B_1$, functorial for regular morphisms on $B_1$ (hence for regular morphisms on $B$), with pullback $X_1'= X_1\times_{B_1} B_1'$, such that $$\Spec_{X_1'}\left(R_J\otimes \cO_{X_1'}/B_1'\text{-torsion }\right) =: \Spec_{X_1'}\left( \oplus J_i'\right)$$  is flat over  $B_1'$.   The induced homomorphism $J_i' \to \cO_{X_1'}$ remains injective, so $  \oplus J_i'$ remains a finitely generated Rees algebra representing a $\QQ$-ideal $J\cO_{X_1'}$.

Functoriality implies that $B'_1 \to B_1$ is an isomorphism in the  loci over which $J\cO_{X_1}$ is flat. This is automatic where $\cI_{\infty}$ is principal.

It follows that $\cO_{X_1'}/J\cO_{X_1'}$ has $B_1'$-torsion in pure codimesnion 1. Call its support $D_{X_1'} \subset X_1'$, with image $D_{B_1'}\subset B_1'$.

Let $B_1' \to B'$ be the relative coarse moduli space over $B$. Write $D_{B'}\subset B'$ for the image of $D_{B_1'}$. In the diagram
$$\xymatrix{B_1'\ar[r]\ar[dr] & B_1 \ar[dr] \\
& B'\ar[r] & B}$$
the horizontal arrows are birational, the diagonal arrow are coarse moduli spaces, but so far $B_1' $ and $ B'$ are not endowed with nice logarithmic structures.

\subsection{Resolution} We endow $B'$ with the boundary divisor coming from $B$, the exceptional of $B' \to B$, and $D_{B'}$. Let $B'' \to B$ be the result of functorial resolution of singularities of $B'$ making the preimage of the divisor normal crossings. Then $B''$ is regular and log regular.

Pulling back $X'' := (X \times _B B'')^\sat$ we have that $X'' \to B''$ is log regular and equidimensional, hence integral.

Pulling back $B_1''=(B_1 \times _B B'')^\sat$ we have that $B_1''$ is log regular and $B_1'' \to B''$ is Kummer-\'etale.

Finally taking $X_1''=(X \times _B B_1'')^\sat = X_1 \times_{B_1} B_1'' $ we have that $X_1'' \to X''$ is Kummer-\'etale, $X_1'' \to B_1''$ log regular and saturated, the $\QQ$-ideal $J\cO_{X_1''}$ is flat over $B_1''$, and  $\cO_{X_1''}/J\cO_{X_1''}$ has \emph{monomial} $B_1''$-torsion in pure codimesnion 1.





 We obtain a diagram \eqref{pullback},
the saturated pullback of  \eqref{saturation}, 
and this diagram has the same properties as \eqref{saturation}. This completes the proof of Lemma \ref{Lem:flattening}. \qed


%
%
%
%
%

\subsection{Proof of monomialization}

We now replace diagram \eqref{saturation} by \eqref{pullback}, and assume   $J_1:= J\cO_{X_1} $ is flat over $B_1$, and the $B_1$-torsion of $\cO_{X_1}/J\cO_{X_1}$ is  monomial in pure codimension 1. At the same time we have guaranteed that $X\to B$ is still log regular, 
$B$ is regular,  and $B_1 \to B$ is Kummer.

We now claim:


\begin{lemma}\label{Lem:submonomial}

$J_1$ is a submonomial center.
\end{lemma}

\begin{proof}

{\sc Verification of axioms.}
This can be checked on formal completion at any geometric point  $p\in X_1$ lying over  a geometric point  $b\in B_1$.

By \cite[Lemma 3.1.5]{ATW-relative} we may test that $J_1$ is submonomial after pulling back $X_1$ and $J_1$ via an \'etale morphism $\tilde B_1 \to B_1$ with $b$ in its image, hence we may assume the logarithmic structure at $b$ is Zariski.

By properness $X_1, J_1$ are locally  of finite presentation over $B$. So we may replace $b$ and $B_1$ so that the residue field at $b$ is algebraically closed.

We may now verify Axioms A1--A7 of  \cite[Sections 3.3-3.4]{ATW-relative} hold, so that we may apply results of that section. We denote by $P$ the characteristic monoid at $b$ and by $Q$ the characteristic monoid at $p$.
\begin{itemize}
\item[A1.] $X_1 \to B_1$ is a logarithmic orbifold since it is logarithmically regular, and the property of having enough derivation lifts from $X \to B$.
\item[A2.] $B\in \mathbf{B}$, namely it is logarithmically regular, by assumption.
\item[A3.] $X_1 \to B_1$ has abundance of derivations since $X \to B$ does, by \cite[Proposition 2.7.2]{ATW-relative}.
\item[A4.] The residue field of $b$ is algebraically closed as it is a geometric point.
\item[A5.] The logarithmic structure at $b$ is Zariski as arranged above.
\item[A6.] $X_1 \to B_1$ is exact as it is  integral.
\item[A7.] $Q^{\gp}/P^{\gp}$ is torsion free as $X_1 \to B_1$ is saturated.
\end{itemize}

%
%
%



{\sc Description of generators.} Following \cite[Section 3.2.1]{ATW-relative},
we use the notation $\tilde Q$ for the image of the monoid $Q$ in $Q^{\gp}/P^{\gp}$. For an element $\tilde q \in \tilde Q$ we write $Q_{\tilde q}$ for its preimage in $Q$.

According to \cite[Proposition 3.4.5 (iv)]{ATW-relative} at any point $x \in X_1$ the ideal $\cI_\infty\hat \cO_{X,x}$ is generated by  finite sums of the form $a^{j}_{\tilde q} = \sum_{q \in Q_{\tilde q} }a^j_q u^q, j=1,\ldots,k_{\tilde q}$, with $a^j_q \in \hat\cO_{B,b}$.

Since $X_1 \to B_1$ is integral, the $P$-set $Q_{\tilde q}$ is freely generated by a single element $q_{\tilde q}$, providing a saturated free $\hat \cO_{B_1,b}$-submodule $\hat \cO_{B_1,b} \cdot u^{q_{\tilde q}} \subset  \hat \cO_{X,x}$. Therefore the $B_1$-torsion of the $\hat \cO_{X,x}$-module $ \hat \cO_{X,x} /\cI_\infty\hat \cO_{X,x} $
is generated by  its submodules $$\hat \cO_{B_1,b} \cdot u^{q_{\tilde q}}\ /\ (a^{j}_{\tilde q},j=1,\ldots,k_{\tilde q})\quad  \subset\quad  \hat \cO_{X,x}  /\cI_\infty\hat \cO_{X,x} .$$

Being in pure codimension 1 (equivalently, $\cI_\infty$ being flat) means that the submonoid generated by $(a^{j}_{\tilde q})$ is free of rank 0 or 1, and we may assume $k_{\tilde q}\leq 1$, so  at most one generator $$a_{\tilde q}=a^{1}_{\tilde q} = \sum_{q \in Q_{\tilde q} }a_q u^q$$ suffices for each $\tilde q$. Write $a'_{\tilde q}= \sum_{q \in Q_{\tilde q} }a_q u^{q-q_{\tilde q}}$ so that $a_{\tilde q} = a'_{\tilde q}\cdot u^{q_{\tilde q}}$. Noting that $q-q_{\tilde q} \in P$ we have $a'_{\tilde q} \in \hat\cO_{B_1,b}$.

Its divisor satisfies  $\Div(a'_{\tilde q})\subset D_{B_1}$.  Since $D_{B_1}$ is a \emph{monomial} divisor we  have $\Div(a'_{\tilde q})$ is monomial, namely $(a'_{\tilde q}) = (u^{q_0})$ for some $q_0 \in P$. This means that $\cI_\infty$ is monomial, so $J_1$ is sumbonomial.
\end{proof}

To complete the proof of Proposition \ref{Prop:monomialization} we note that, by \cite[Lemma 3.1.6]{ATW-relative}, the $\QQ$-ideal  $J$ is submonomial as well, as needed.\qed

\bibliographystyle{amsalpha}
\bibliography{principalization}

\providecommand{\bysame}{\leavevmode\hbox to3em{\hrulefill}\thinspace}
\providecommand{\MR}{\relax\ifhmode\unskip\space\fi MR }
\providecommand{\MRhref}[2]{%
  \href{http://www.ams.org/mathscinet-getitem?mr=#1}{#2}
}
\providecommand{\href}[2]{#2}
\begin{thebibliography}{AdSTW25}

\bibitem[AdSTW25]{ABTW-foliated}
Dan Abramovich, André~Belotto da~Silva, Michael Temkin, and Jarosław
  Włodarczyk, \emph{Principalization on logarithmically foliated orbifolds},
  2025, \url{https://arxiv.org/abs/2503.00926}.

\bibitem[AK00]{AK}
Dan Abramovich and Kalle Karu, \emph{Weak semistable reduction in
  characteristic 0}, Invent. Math. \textbf{139} (2000), no.~2, 241--273.
  \MR{1738451 (2001f:14021)}

\bibitem[ATW20]{ATW-relative}
Dan Abramovich, Michael Temkin, and Jarosław Włodarczyk, \emph{Relative
  desingularization and principalization of ideals}, 2020,
  \url{https://arxiv.org/abs/2003.03659}.

\bibitem[McQ24]{McQuillan-flat}
Michael McQuillan, \emph{Flattening and algebrisation}, 2024,
  \url{https://arxiv.org/abs/2412.00998}.

\bibitem[{Mol}16]{Molcho}
Sam {Molcho}, \emph{{Universal Weak Semistable Reduction}}, 2016,
  \url{https://arxiv.org/abs/1601.0030}.

\bibitem[Que20]{Quek}
Ming~Hao Quek, \emph{Logarithmic resolution via weighted toroidal blow-ups},
  2020, \url{https://arxiv.org/abs/2005.05939}.

\bibitem[Ryd25]{Rydh-flattening}
David Rydh, \emph{Functorial flatification of proper morphisms}, 2025,
  \url{https://arxiv.org/abs/2501.08394}.

\end{thebibliography}

\end{document}